\documentclass[11pt]{amsart}
\usepackage{amssymb, latexsym, amsmath, amsfonts, tikz}
\usepackage{hyperref, enumitem, fancyhdr}
\usepackage{color}
\usepackage{tikz-cd} 

\newtheorem{thm}{Theorem}[section]

\newtheorem{lem}[thm]{Lemma}

\theoremstyle{definition}

\theoremstyle{remark}

\numberwithin{equation}{section}
\theoremstyle{remark}

\newcommand{\mbb}{\mathbb}

\newcommand{\sm}{\setminus}

\newcommand{\no}{\noindent}

\begin{document}
\title{Relation between H\'{e}non maps with biholomorphic escaping sets} 
\author{Ratna Pal}

\address{Indian Institute of Science Education and Research Mohali, Knowledge City, Sector -81, Mohali, Punjab-140306, India}
\email{ratna.math@gmail.com, ratnapal@iisermohali.ac.in}

\begin{abstract}
Let $H$ and $F$ be  two H\'{e}non maps with biholomorphically equivalent escaping sets, then there exist affine automorphisms $A_1$ and $A_2$ in $\mathbb{C}^2$ such that 
\[
F=A_1\circ H \circ A_2
\] 
in $\mathbb{C}^2$.
\end{abstract}

\maketitle

\section{Introduction}
In the complex plane the simplest examples of holomorphic dynamical systems with non-trivial dynamical behaviour are the polynomial maps of degree greater than or equal to $2$.  The linear polynomials, in other words, the automorphisms of the complex plane, generate trivial dynamics. In contrast,  the class of polynomial automorphisms of $\mathbb{C}^2$ is large and possesses rich dynamical features. A dynamical classification of these maps was given by Friedland--Milnor \cite{FM}. They showed that any polynomial automorphism of $\mathbb{C}^2$ is conjugate to one of the following maps: 

$\bullet$ an {\it affine map};

$\bullet$ an {\it elementary map}, i.e., the maps of the form $(x,y)\mapsto (ax+b, sy+p(x))$ with $as\neq 0$, where $p$ is a polynomial in single variable of degree strictly greater than one;

$\bullet$ a finite composition of H\'{e}non maps, where H\'{e}non maps are the maps of the form
\begin{equation}\label{S Henon}
(x,y)\mapsto (y, p(y)-\delta x)
\end{equation}
 with $\delta\neq 0$ and $p$ a polynomial in single variable of degree $d\geq 2$. 
 
 The degree of a single H\'{e}non map $H$ of the form (\ref{S Henon}) is defined to be the degree of the polynomial $p$.
The degree of composition of H\'{e}non maps $H_n\circ \cdots \circ H_1$ is defined to be $d_n\cdots d_1$ where $d_i=\deg(H_i)$, for $1\leq i \leq n$.

H\'{e}non maps are generalization of classical real quadratic H\'{e}non maps introduced by astronomer Michel H\'{e}non. The relevance of these maps in complex dynamics became apparent in the above-mentioned classification theorem of Friedland--Milnor.  Moreover it turned out that the composition of H\'{e}non maps are the only dynamically non-trivial polynomial automorphisms in $\mathbb{C}^2$, which naturally drew attention of many foremost researchers towards these maps. 
The pioneering work on H\'{e}non maps was done  by  Bedford--Smillie (\cite{BS1}, \cite{BS2}, \cite{BS3}), Forn\ae ss--Sibony (\cite{FS}) and Hubbard--Oberste-Vorth (\cite{HO}, \cite{HO2}).

As in the case of polynomials in the complex plane, the orbit of any point in $\mathbb{C}^2$ under the iterations of a H\'{e}non map $H$ (or more generally, finite composition of H\'{e}non maps) either diverges to infinity or always remains bounded. The collection of points $I_H^+ \subseteq \mathbb{C}^2$ which escape to infinity are called the {\it{escaping set}} of $H$ and the collection of points $K_H^+\subseteq \mathbb{C}^2$ whose orbits remain bounded are called  the {\it{non-escaping set}} of $H$. The union of escaping set $I_H^+$ and the interior of non-escaping set $K_H^+$ is the largest set of normality of the sequence of maps $\{H^n\}_{n\geq 1}$, that is, $I_H^+ \cup {\rm{int}}K_H^+$ is the Fatou set of $H$. The escaping sets and the interiors of non-escaping sets of H\'{e}non maps can be thought of as analogues of the unbounded components and the bounded components of Fatou sets of polynomials in the complex plane. However, the non-escaping sets of H\'{e}non maps are not bounded in $\mathbb{C}^2$. The common boundary set $J_H^+$ of the escaping set $I_H^+$ and the non-escaping set $K_H^+$ is the Julia set of $H$. 

The present article addresses a {\it rigidity} property of H\'{e}non maps of the form (\ref{S Henon}).
{\it To what extent do the escaping sets of H\'{e}non maps determine the H\'{e}non maps? }In other words,  if the escaping sets of a pair of H\'{e}non maps $H$ and $F$ of degree $d$ are biholomorphically equivalent, then are these two H\'{e}non maps closely related? 

This question is first studied in a recent work of Bonnot--Radu--Tanase (\cite{BRT}), where they prove that $H$ and $F$ coincide, for $d=2$. Further, they produce examples to show that for $d\geq 3$, $H$ and $F$ might not be even conjugate to each other. In this article we establish a precise relation between $H$ and $F$ of any degree $d\geq 2$ with biholomorphic escaping sets (Theorem \ref{main'}).

The rigidity question raised here is conceived based on an explicit description of analytic structure of the escaping sets given by Hubbard--Oberste-Vorth in \cite{HO}. A convenient description of the escaping set $I_H^+$ of a given H\'{e}non map $H$ is given in terms of logarithmic rate of escape function, the so-called Green's function $G_H^+$  of the H\'{e}non map $H$. One can show that $I_H^+$ is precisely where the Green's function is strictly positive. Further, $G_H^+: I_H^+ \rightarrow \mathbb{R}_+$ is a pluri-harmonic submersion and the level sets of $G_H^+$ are three dimensional manifolds, which are naturally foliated by copies of  $\mathbb{C}$. The Green's function $G_H^+$ is inextricably related to the B\"{o}ttcher function $\phi_H^+$ of $H$, which is one of the key ingredients in describing the analytic structure of $I_H^+$.  The B\"{o}ttcher functions of H\'{e}non maps can be considered as analogues of B\"{o}ttcher functions of polynomials in $\mathbb{C}$ near infinity and they are defined in appropriate neighbourhoods of a point at infinity $[0:1:0]$ in $\mathbb{P}^2$.  To be a bit more precise, for $R>0$ sufficiently large, $\phi_H^+$ is defined on the open set $V_R^+=\left\{(x,y)\in \mathbb{C}^2: \lvert y\rvert> \max\left \{\lvert x\rvert, R\right\}\right \} \subseteq I_H^+$ by means of approaching the point $[0:1:0]$ by the n-fold iteration $H^n$, then returning back by appropriate $d^n$-th root of the mapping $y\mapsto y^{d^n}$ and finally taking the limit as $n\rightarrow \infty$.  Consequently, the range of $\phi_H^+$ lies inside $\mathbb{C}\setminus \bar{\mathbb{D}}$, where $\mathbb{D}$ is the unit disc in $\mathbb{C}$ and $G_H^+\equiv \log \left\lvert \phi_H^+\right\rvert$ in $V_R^+$.
Although the B\"{o}ttcher function $\phi_H^+$ does not extend analytically to $I_H^+$, it extends along curves in $I_H^+$ starting in $V_R^+$ and defines a multi-valued analytic map in $I_H^+$. Let $\tilde{I}_H^+$ be the covering of $I_H^+$ obtained as the Riemann domain of $\phi_H^+$ and let $\phi_H^+$  lifts as a single-valued holomorphic function $\tilde{\phi}_H^+: \tilde{I}_H^+ \rightarrow \mathbb{C}\setminus \bar{\mathbb{D}}$. Further, $\tilde{I}_H^+$ is biholomorphically equivalent to the domain $\{(z,\zeta): z\in \mathbb{C}, \lvert \zeta \rvert>1\}$. By construction $\zeta=\tilde{\phi}_H^+$ and thus in this new coordinate the level sets of $\tilde{\phi}_H^+$ simply straightens out.  
The H\'{e}non map $H$ lifts as a map $\tilde{H}: \mathbb{C}\times (\mathbb{C}\setminus \bar{\mathbb{D}})\rightarrow  \mathbb{C}\times (\mathbb{C}\setminus \bar{\mathbb{D}})$ and one can write down $\tilde{H}$ explicitly (see (\ref{H tilde})). 

The following result relies on two main ingredients, also used by Bonnot--Radu--Tanase in \cite{BRT}: The above-mentioned  explicit description of the covering $\tilde{I}_H^+$ of $I_H^+$ and a method given by Bousch in \cite{bousch}.

\begin{thm}\label{main}
Let $H(x,y)=(y,p_H(y)-\delta_H x)$ and $F(x,y)=(y,p_F(y)-\delta_F x)$ be  a pair of H\'{e}non maps, where $p_H(y)=y^d+\sum_{i=0}^{d-2} a_{i}^H y^i$ and $p_F(y)=y^d+\sum_{i=0}^{d-2} a_{i}^F y^i$. Let $I_H^+$ and $I_F^+$ be escaping sets of $H$ and $F$, respectively and let $I_H^+$ and $I_F^+$ be biholomorphically equivalent. Then $\beta p_H(y)=\alpha p_F(\alpha y)$, for some $\alpha, \beta \in \mathbb{C}$ with $\alpha^{d+1}=\beta$ and $\beta^{d-1}=1$. Further, we have $\delta_H=\gamma \delta_F$, with $\gamma^{d-1}=1$. Therefore, 
\begin{equation}
F\equiv L \circ B \circ H \circ B,
\end{equation}
where $B(x,y)=(\gamma \alpha \beta^{-1}x, \alpha^{-1}y)$ and $L(x,y)=(\gamma^{-1}\beta x,\beta y )$, for all $(x,y)\in \mathbb{C}^2$.
\end{thm}
Now note that $p_H$ and $p_F$ in Theorem \ref{main} are monic and centered (next to highest coefficients vanish), whereas for  an arbitrary H\'{e}non map the associated polynomial in one variable is   not necessarily monic and centered. However, it is not hard to see that up to conjugation any arbitrary H\'{e}non map is of the form as in Theorem \ref{main}. Here goes a brief justification.  Up to conjugation by an affine automorphism of $\mathbb{C}$, any polynomial in one variable is a monic and centered polynomial of the same degree.  In particular, there exists an affine automorphism $\sigma_H$ of $\mathbb{C}$ such that $\sigma_H^{-1}\circ {p}_H\circ \sigma_H=\hat{p}_H$ in $\mathbb{C}$, where $\hat{p}_H$ is monic and centered. Thus if we consider the affine automorphism $A_H(x,y)=(\sigma_H(x),\sigma_H(y))$, for $(x,y)\in \mathbb{C}^2$, then $A_H^{-1}\circ {H} \circ A_H=\hat{H}$, where $\hat{H}(x,y)=(y,\hat{p}_H(y)-{\delta}_H x)$, for all $(x,y)\in \mathbb{C}^2$ with $\deg({\hat{p}_H})=\deg(p_H)$.  Clearly,  $A_H(I_H^+)=I_{\hat{H}}^+$. Similarly, there exist an affine map $A_F$ and a H\'{e}non map $\hat{F}$ such that $A_F^{-1}\circ  F \circ A_F=\hat{F}$ and $A_F(I_F^+)=I_{\hat{F}}^+$, where  $\hat{F}(x,y)=(y,\hat{p}_F(y)-{\delta}_F x)$, for all $(x,y)\in \mathbb{C}^2$ with $\hat{p}_F$ monic and centered.  Therefore, once we prove Theorem \ref{main}, the following result is obtained immediately.

\begin{thm}\label{main'}
Let $H$ and $F$ be  two H\'{e}non maps with biholomorphically equivalent escaping sets, then there exist affine automorphisms $A_1$ and $A_2$ in $\mathbb{C}^2$ such that 
\[
F=A_1\circ H \circ A_2
\] 
in $\mathbb{C}^2$.
\end{thm}

\no 
\textbf{Acknowledgement:} The author would like to thank the referees for making helpful comments.

\section{Preliminaries}
\no 
Let 
\begin{equation}\label{Henon}
H(x,y)=(y,p(y)-\delta x)
\end{equation}
be a H\'{e}non map, where $p$ is a monic and centered  polynomial in single variable of degree $d\geq 2$ and $\delta\neq 0$.  In this section, we see a few fundamental definitions and a couple of known results on H\'{e}non maps pertaining to the theme of the present article.

\medskip
\no 
{{\it {Filtration}}: For $R>0$, let
\begin{align*}
V^+_R &= \{ (x,y) \in \mathbb C^2: \vert x \vert < \vert y \vert, \vert y \vert > R \},\\
V^-_R &= \{ (x,y) \in \mathbb C^2: \vert y \vert < \vert x \vert, \vert x \vert > R \},\\
V_R &= \{ (x, y) \in \mathbb C^2: \vert x \vert, \vert y \vert \le R \}.
\end{align*}
This is called a filtration.  For a given H\'{e}non map $H$,  there exists $R > 0$ sufficiently large such that
\[
H(V^+_R) \subset V^+_R, \; H(V^+_R \cup V_R) \subset V^+_R \cup V_R
\]
and
\[
H^{-1}(V^-_R) \subset V^-_R, \; H^{-1}(V^-_R \cup V_R) \subset V^-_R \cup V_R.
\]

\medskip
\no 
{\it{Escaping and non-escaping sets}}: The set 
\[
I^{+}_H =\left \{(x, y) \in \mathbb C^2 : \lVert H^n(x,y)\rVert \rightarrow \infty \text{ as } n\rightarrow \infty\right\}
\]
is called the {\it escaping set} of $H$ and the set
\[
K_H^+ = \left\{(x, y) \in \mathbb C^2 :\;\text{the sequence}\; \left\{ H^{n}(x, y) \right\}_{n\geq 1} \; \text{is bounded} \right\}
\]
is called the {\it non-escaping set} of $H$. One can prove that $K_H^+ \subset V_R \cup V^-_R$ and
\begin{equation}\label{escape}
I_H^+=\mathbb C^2 \sm K^{+}_H = \bigcup_{n=0}^{\infty} H^{- n}(V^{+}_R).
\end{equation}
Any H\'{e}non map $H$ can be extended meromorphically to $\mathbb{P}^2$ with an isolated indeterminacy point  $[0:1:0]$. In fact, one can prove that the points in $I_H^+$ under iteration of $H$ converges to the point $[0:1:0]$ uniformly (on compacts). 

\medskip
\no 
{\it{Green's function}}:  The {\it Green's function} of $H$ is defined to be
\[
G^{+}_H(x, y) := \lim_{n \rightarrow \infty} \frac{1}{d^n} \log^+ \left\Vert H^{ n}(x, y)\right \Vert,
\]
for all $(x,y)\in \mathbb{C}^2$, where $\log^+(t)=\max\{\log t,0\}$. It turns out that  $G_H^+$ is non-negative everywhere in $\mathbb{C}^2$,  plurisubharmonic in $\mbb C^2$,  pluriharmonic on $\mbb C^2 \sm K^{+}_H$ and vanishes precisely on $K^{+}_H$. Further, 
\[
G^{+}_H \circ H = d G^{+}_H
\]
in $\mathbb{C}^2$. The Green's function $G^{+}_H$ has logarithmic growth near infinity, i.e., there exist $R>0$ and $L>0$ such that   
\begin{equation*}
\log^+ \lvert y \rvert-L\leq G_H^+ (x,y) \leq \log^+ \lvert y \rvert+L,
\end{equation*}
for $(x,y)\in \overline{V_R^+ \cup V_R}$.

\medskip
\no 
{\it{B\"{o}ttcher function}}: Let $H^n(x,y)=\left({(H^{n})}_1(x,y),{(H^{n})}_2(x,y)\right)$, for $(x,y)\in \mathbb{C}^2$. Note that $y_n={(H^{n})}_2(x,y)$ is a polynomial in $x$ and $y$ of degree $d^{n}$. The function
\[
\phi_H^+(x,y):=\lim_{n\rightarrow \infty}y_{n}^{\frac{1}{d^{n}}}=y.\frac{{y_1}^{\frac{1}{d}}}{y}.\cdots.\frac{y_{n+1}^{\frac{1}{d^{n+1}}}}{y_{n}^{\frac{1}{d^{n}}}}\cdots,
\]
for $(x,y)\in V_R^+$, defines an analytic function from $V_R^+$ to $\mathbb{C}\setminus \bar{\mathbb{D}}$ and is called the {\it B\"{o}ttcher function} of the H\'{e}non map $H$. Further, 
\[
\phi_H^+\circ H(x,y)={\left(\phi_H^+(x,y)\right)}^d,
\]
 for all $(x,y)\in V_R^+$ and
\[
\phi_H^+(x,y)\sim y \text{ as } \lVert(x,y)\rVert\rightarrow \infty  
\]
 in  $V_R^+$.
Comparing the definitions of Green's function and B\"{o}ttcher functions it follows  instantly that
\begin{equation*}\label{Green1}
G_H^+=\log \lvert\phi_H^+\rvert
\end{equation*}
in $V_R^+$.

For a detailed treatment of the above discussion, the inquisitive readers can see \cite{BS1}, \cite{BS2}, \cite{BS3}, \cite{HO} and \cite{MNTU}.

\medskip 
\no 
Now we present one of the technical ingredients required for the proof of Theorem \ref{main}. A series of change of coordinates we see here is a part of the standard theory of H\'{e}non maps. However, most of the results presented here are paraphrasing of a couple of lemmas appearing in the beginning of the appendix of \cite{BRT}. However the genesis of these results goes back to \cite{Favre} and  \cite{HO}.

For $M>0$ sufficiently large, let
\[
U_R^+=\left\{(x,y)\in V_R^+: \lvert\phi_H^+(x,y)\rvert >  M \max\{R, \lvert x \rvert \}\right\}. 
\]
One can easily check  that $U_R^+ \subseteq V_R^+$ and $H(U_R^+)\subseteq U_R^+$. 

\begin{lem}(\cite[Lemma 7.3.7] {MNTU}, \cite[Prop.\ 2.2]{Favre}) \label{lem MNTU F}
There exists a holomorphic function $\psi_H$ on $U_R^+$ such that 
\begin{itemize}
\item 
$\psi_H \circ H (x,y)=( {\delta}/{d}) \psi_H(x,y) + Q \left(\phi_H^+(x,y)\right)$,
for all $(x,y)\in U_R^+$, where $Q$ is a monic polynomial of degree $d+1$;
\item 
the map $\Phi_H=(\psi_H, \phi_H^+):U_R^+ \rightarrow \mathbb{C}^2$  is an injective holomorphic map.
\end{itemize}
\end{lem}
\no
The map $(x,y) \mapsto (x,\zeta)=(x,\phi_H^+(x,y))$ maps $U_R^+$ biholomorphically to 
$\left\{(x,y)\in \mathbb{C}^2: \lvert y\rvert >  M \max\{R, \lvert x \rvert \}\right\}.$ Let $(x,\zeta) \mapsto (x, y(x,\zeta))$ be the inverse map. It turns out that 
$$
\psi_H(x,y)=\zeta \int_0^{x} \frac{\partial y}{\partial \zeta}(u,\zeta)du,
$$
for $(x,y)\in U_R^+$ and $Q(\zeta)$ is the polynomial part of the power series expansion of 
\begin{equation}\label{Q'}
\zeta^d \int_0^{y(0,\zeta)} \frac{\partial y}{\partial \zeta}(u,\zeta^d)du
\end{equation}
(see proof of \cite[Lemma 7.3.7] {MNTU} for details).


Next we consider the following change of coordinate near $p=[0:1:0]$:
\[
T:\left(x,y(x,\zeta)\right)=(x,y) \mapsto \left (\frac{x}{y}, \frac{1}{y}\right)=(t,w). 
\]
Note that $T^2=\rm{Id}$ and $H$ takes the following form in $(t,w)$-coordinate:
\begin{align*}
(t,w) \mapsto (x,y) \mapsto (y, p(y)-\delta x) \mapsto \left(\frac{y}{p(y)-\delta x}, \frac{1}{p(y)-\delta x}\right)\\
=\left( \frac{w^{d-1}}{w^d p(1/w)-\delta t w^{d-1}},\frac{w^{d}}{w^d p(1/w)-\delta t w^{d-1}}\right). 
\end{align*}
\begin{lem} \label{lem zeta inv}
With the above notations, near $p=[0:1:0]$ we have 
\begin{equation}\label{zeta inv}
 {1}/{\zeta}= w(1+ w \alpha(t,w)),
\end{equation}
where $\alpha(t,w)$ is a power series in $t, w$.
\end{lem}
\begin{proof}
Note that 
\[
\frac{1}{\zeta}=\frac{1}{\phi_H^+(x,y)}=\lim_{n\rightarrow \infty}{\left [\frac{1}{{(H^n(x,y))}_2}\right]}^{\frac{1}{d^n}}=\lim_{n\rightarrow \infty}{\left[{(TH^n T^{-1}(t,w))}_2\right]}^{\frac{1}{d^n}}.
\]
Equivalently,  
\begin{eqnarray*}
\frac{1}{\zeta}&=& w. {\left[\left(\frac{w^d}{w^d p(\frac{1}{w})-\delta t w^{d-1}}\right)/{w^d}\right]}^{\frac{1}{d}}.
{\left[\left(\frac{w_1^d}{w_1^d p(\frac{1}{w_1})-\delta t_1 w_1^{d-1}}\right)/{w_1^d}\right]}^{\frac{1}{d^2}}\cdots \\
&=& w. {\left[\frac{1}{w^d p(\frac{1}{w})-\delta t w^{d-1}}\right]}^{\frac{1}{d}}.{\left[\frac{1}{w_1^d p(\frac{1}{w_1})-\delta t_1 w_1^{d-1}}\right]}^{\frac{1}{d^2}} \cdots =w X(t,w)
\end{eqnarray*}
(here the $d^n$-th roots, for $n\geq 1$, are taken to be the principal branches of roots). One can check  that the above series converges. Further, since $\left\lvert {\phi_H^+(x,y)}/{y}\right \rvert$ is bounded in $V_R^+$, it follows that $X(t,w)$ is also bounded in its domain of definition, which is a subset of  $\{\lvert t\rvert<1\}\times \left\{0< \lvert w\rvert<{1}/{R}\right\} $. Thus $X(t,w)$ has a power series expansion. Also note that $X(t,0)=1$, for all $t$. So  $X(t,w)=1+ w \alpha(t,w)$, where $\alpha(t,w)$ is a power series in $t$ and $w$. Therefore, finally we obtain (\ref{zeta inv}).
\end{proof}

Next we consider the following change of variables: $(t,w) \mapsto (r,s)= ({x}/{\zeta}, {1}/{\zeta})$, which  is a biholomorphism and is tangent to the identity (see \cite[Lemma 4.4]{BRT}). Therefore, 
\[
t=r+ T_2(r,s) \text{ and } w=s+ S_2(r,s),
\]
where $T_2$ and  $S_2$ both are power series with monomials of degree at least $2$. 
\begin{lem}
With the above notations, we have 
\begin{eqnarray*}
y&=& \zeta \left(1+{C}/{\zeta}+ U\left({x}/{\zeta}, {1}/{\zeta}\right)\right),
\end{eqnarray*}
where $C\in \mathbb{C}$ and  $U$ is a power series in two variables with monomials of degree at least $2$. 
\end{lem}
\begin{proof}
By Lemma \ref{lem zeta inv}, we have ${1}/{\zeta}= w(1+w \alpha(t,w))$. Thus  ${y}/{\zeta}= (1+w \alpha(t,w))$. Therefore, by replacing $t=r+ T_2(r,s) \text{ and } w=s+ S_2(r,s)$, we get
\begin{eqnarray*}
y&=& \zeta \left(1+(s+ S_2(r,s))(\alpha(r+ T_2(r,s), s+ S_2(r,s) ))\right)\\
&=&\zeta \left(1+s \beta(r,s)+ S_2(r,s) \beta(r,s) \right), 
\end{eqnarray*}
where $\beta(r,s)=\alpha(r+ T_2(r,s), s+ S_2(r,s))$.
Now since $(r,s)= ({x}/{\zeta}, {1}/{\zeta})$, we have 
\[
y= \zeta \left(1+{C}/{\zeta}+ U\left({x}/{\zeta}, {1}/{\zeta}\right)\right),
\]
where $C\in \mathbb{C}$ is the constant term of the power series expansion of $\beta$ and $U$ is a power series in two variables without constant term. 
\end{proof}
\no 
Therefore, 
\begin{equation*}
y(x,\zeta)=\left(\zeta+C+\frac{A^0_{-1}}{\zeta}+\cdots\right)+x\left(\frac{A_{-1}^1}{\zeta^{m_1}}+\cdots\right)+
x^2\left(\frac{A_{-1}^2}{\zeta^{m_2}}+\cdots\right)+ \cdots+ x^n\left(\frac{A_{-1}^n}{\zeta^{m_n}}+\cdots\right)+\cdots,
\end{equation*}
where $m_1=1$, $m_n=n-1$, for $n\geq 2$ and $A_{-1}^0, A_{-1}^1,A_{-1}^2$ and so on are in $\mathbb{C}$. 
\medskip 

The following lemma is one of the main takeaways from this section.
\begin{lem}\label{yzetaQ}
With the above notations, we have 
\[
y(0,\zeta)=\zeta+ \frac{D_1}{\zeta}+ \frac{D_2}{\zeta^2}\cdots,
\]
and 
\[
\zeta(0,y)=y+ \frac{L_1}{y}+ \frac{L_2}{y^2}\cdots,
\]
where $D_i$'s and $L_i$'s are constants. Further, $Q(\zeta)$ is given by the polynomial part of $\zeta^d y(0,\zeta)$ and therefore, $Q$ is centered and monic. 
\end{lem}
\begin{proof}
Since $p$ is monic and centered in (\ref{Henon}),  it follows from \cite[Prop.\ 5.2]{HO} that
\[
\zeta(0,y)=y+ \frac{L_1}{y}+ \frac{L_2}{y^2} \cdots.
\]
Now since $y=y(0,\zeta(0,y))$, 
\begin{eqnarray*}
y=y(0,\zeta(0,y))&=& \zeta(0,y)+D_0+ \frac{D_1}{\zeta(0,y)}+\cdots\\
&=& \left(y+\frac{L_1}{y}+\frac{L_2}{y^2}+\cdots\right)+D_0+ D_1{\left(y+\frac{L_1}{y}+\frac{L_2}{y^2}+\cdots\right)}^{-1}+\cdots.
\end{eqnarray*}
Thus $D_0=0$ and consequently, we have
\[
y(0,\zeta)=\zeta+\frac{D_1}{\zeta}+ \frac{D_2}{\zeta^2}+\cdots.
\]
The last assertion follows immediately from (\ref{Q'}).
\end{proof}
As we shall see in Section 3 and in Section 4, the particular forms of the  power series representations of $y(0,\zeta)$ and $\zeta(0,y)$ as obtained in Lemma \ref{yzetaQ} and the knowledge of analytic structure of the escaping sets of H\'{e}non maps are the key players in unravelling the relation between any pair of H\'{e}non maps with biholomorphic escaping sets. 

Here we give a brief description of analytic structure of H\'{e}non maps. For a detailed account of analytic structure of escaping sets of H\'{e}non maps the inquisitive readers can look at  \cite{Favre}, \cite{HO} and \cite{MNTU}.  For any H\'{e}non map $H$, it turns out that the Riemann surface of the B\"{o}ttcher function $\phi_H^+$ is a  covering space of the escaping set $I_H^+$ and it is isomorphic to $\mathbb{C}\times (\mathbb{C}\setminus \bar{\mathbb{D}})$. For a H\'{e}non map $H$ of degree $d$,  the fundamental group of $I_H^+$ is 
\[
\mathbb{Z}\left[{1}/{d}\right]=\left \{{k}/{d^n}: k,n \in \mathbb{Z}\right \}
\]
and the covering $\mathbb{C}\times (\mathbb{C}\setminus \bar{\mathbb{D}})$ of $I_H^+$ arises corresponding to  the subgroup $\mathbb{Z} \subseteq \mathbb{Z}\left[{1}/{d}\right]$. Therefore, $I_H^+$ is a quotient of $\mathbb{C}\times (\mathbb{C}\setminus \bar{\mathbb{D}})$ by some discrete subgroup of the automorphisms of $\mathbb{C}\times (\mathbb{C}\setminus \bar{\mathbb{D}})$ isomorphic to ${\mathbb{Z}\left[{1}/{d}\right]}/{\mathbb{Z}}$.  For each element $\left[{k}/{d^n}\right] \in \mathbb{Z}[{1}/{d}]/\mathbb{Z}$, there exists a unique deck  transformation $\gamma_{{k}/{d^n}}$ from  $\mathbb{C}\times (\mathbb{C}\setminus\bar{\mathbb{D}})$ to $\mathbb{C}\times (\mathbb{C}\setminus\bar{\mathbb{D}})$  such that 
\begin{equation}\label{fiber form}
\gamma_{{k}/{d^n}}
\begin{bmatrix}
z\\ \zeta
\end{bmatrix}
=\begin{bmatrix} 
 z+ \frac{d}{\delta} \sum_{l=0}^{n-1} {\left(\frac{d}{\delta}\right)}^l\left(Q(\zeta^{d^l})-Q\left( {\left(e^{\frac{2 k\pi i}{d^n}}\zeta)\right)}^{d^l} \right)\right)\\
 e^{\frac{2 k\pi i}{d^n}}\zeta
 \end{bmatrix},
\end{equation}
for $n\geq 0$ and $k\geq 1$. Further, 
$H$ lifts as a holomorphic map 
\begin{equation} \label{H tilde}
\tilde{H}(z,\zeta)=\left(({\delta}/{d})z+Q(\zeta), {\zeta}^d \right)
\end{equation}
from $\mathbb{C}\times (\mathbb{C}\setminus \bar{\mathbb{D}})$ to $\mathbb{C}\times (\mathbb{C}\setminus \bar{\mathbb{D}})$.

\section{A brief idea of the proof of the main theorem}
\no 
In this section we sketch the main idea of the proof of  Theorem \ref{main}. 

\medskip 
\no 
Let $H(x,y)=(y,p_H(y)-\delta_H x)$ and $F(x,y)=(y,p_F(y)-\delta_F x)$ be  a pair of H\'{e}non maps of degree $d$. Thus the fundamental groups of both $I_H^+$ and $I_F^+$ are  $\mathbb{Z}[{1}/{d}]$. Let $a$ be a biholomorphism between $I_H^+$ and $I_F^+$.   Any biholomorphism between $I_H^+$ and $I_F^+$, which induces identity as an isomorphism between the fundamental groups of $I_H^+$ and $I_F^+$, can be lifted as a biholomorphism from $\mathbb{C}\times (\mathbb{C}\setminus \bar{\mathbb{D}})$ to $\mathbb{C}\times (\mathbb{C}\setminus \bar{\mathbb{D}})$.  Now since any group isomorphism of $\mathbb{Z}[{1}/{d}]$ is of the form $x\mapsto \pm d^s x$, for some $s\in \mathbb{Z}$, the map $a$ up to pre-composition with some $n$-fold iterates of $F$ (or $F^{-1}$), i.e.,  the map $F^{\pm n}\circ a$, for some $n\in \mathbb{N}$, induces identity map between the fundamental groups of the escaping sets $I_H^+$ and $I_F^+$. Thus, it is harmless to assume that $a$ lifts  as a biholomorphism $A$ from  $\mathbb{C}\times (\mathbb{C}\setminus \bar{\mathbb{D}})$ to  $\mathbb{C}\times (\mathbb{C}\setminus \bar{\mathbb{D}})$. Therefore, $\pi \circ A=a\circ \pi$ and consequently, the fiber of any point $p\in I_H^+$ of the natural projection map $\pi_H: \mathbb{C}\times (\mathbb{C}\setminus \bar{\mathbb{D}}) \rightarrow I_H^+$ maps into  the fiber of the point $a(p)\in I_F^+$ of the projection map $\pi_F: \mathbb{C}\times (\mathbb{C}\setminus \bar{\mathbb{D}}) \rightarrow I_F^+$ by the biholomorphism $A$. Fiber of any point of the projection maps $\pi_H$ and $\pi_F$ are captured by the group of deck transformation of $\mathbb{C}\times (\mathbb{C}\setminus \bar{\mathbb{D}})$, which is isomorphic to $\mathbb{Z}[{1}/{d}]/\mathbb{Z}$.
Thus, it turns out that if $(z,\zeta)\in \mathbb{C}\times (\mathbb{C}\setminus \bar{\mathbb{D}})$ is in the fiber of any point $p\in I_H^+$ (or $I_F^+$), then any other point in the same fiber will be of the form $\gamma_{{k}/{d^n}}(z,\zeta)$, for $k\geq 1$ and $n\geq 0$,  where $\gamma_{{k}/{d^n}}$ is the deck transformation corresponding to $[{k}/{d^n}]\in \mathbb{Z}[{1}/{d}]/\mathbb{Z}$. An explicit description of this fibers can be written down (see (\ref{fiber form})). Now adapting an idea of Bousch (\cite{bousch}), one can narrow down the possible forms of $A$ by comparing the fibers of $p$ and $a(p)$. In fact, in our case, $A$ has a very simple form $(u,v)\mapsto (u+\gamma, \alpha v)$, with $\lvert \alpha\rvert=1$ and $\gamma\in \mathbb{C}$. Thus we obtain an explicit expression for the  image of the foliation $\tilde{\phi}_H^+=\zeta_0$  under the map $A$. These foliations come down to the corresponding escaping sets and induce rigidity on escaping sets. As a consequence, one expects a close relation between $H$ and $F$, which is  validated in our main theorem.

The above-mentioned idea was employed by  Bonnot--Radu--Tanase in \cite{BRT} in establishing the relation between $H$ and $F$ with biholomorphic escaping sets when $\deg(H)=\deg(F)=2$. 
 They attempts to extract the relation between the coefficients of $p_H$ and $p_F$ by directly comparing the fibers of $p$ and $a(p)$. Although their approach works for lower degrees (for $d=2,3$), since precise computations can be carried out in these cases but as degree increases, to extract the relation between the coefficients of $p_H$ and $p_F$ seems very difficult by performing such direct calculations.

We take a different approach towards this problem. For two H\'{e}non maps $H$ and $F$ with biholomorphic escaping sets, we first establish the relation between coefficients of $Q_H$ and $Q_F$ (see Section 3).  The explicit relation between $Q_H$ and $Q_F$  is used to obtain a neat relation between the polynomials $p_H$ and $p_F$, namely,  $\beta p_H(y)=\alpha p_F(\alpha y)$, for all $y\in \mathbb{C}$ with $\alpha^{d+1}=\beta$ and $\beta^{d-1}=1$ (see Section 4). To establish the relation between $H$ and $F$, we are yet to investigate the relation between the Jacobians $\delta_H$ and $\delta_F$.  It turns out that $\delta_H=\gamma \delta_F$, with $\gamma^{d-1}=1$, thanks to \cite{BRT}. It is noteworthy that all our calculations in Section 3 and in Section 4 are performed under the assumption that $\delta_H=\delta_F$. In Section 5, we outline how to handle the case when $\delta_H$ and $\delta_F$ are different.




\section{Relation between $Q_H$ and $Q_F$}
\no 
Let $H$ and $F$ be two H\'{e}non maps as in Theorem \ref{main} with bihholomorphic escaping sets $I_H^+$ and $I_F^+$, respectively. As indicated in the Introduction, for now we assume $\delta_H=\delta_F=\delta$. Later (in Section 6), we handle the case when $\delta_H\neq \delta_F$.

\medskip 
\no 
{{\it {Lifting  biholomorphisms between escaping sets:}}
Recall from Section 2 that the fundamental group of escaping set of any H\'{e}non map of degree $d$ is $\mathbb{Z}[{1}/{d}]$ and $\mathbb{C}\times (\mathbb{C}\setminus \bar{\mathbb{D}})$ is the covering of the escaping set corresponding to the subgroup $\mathbb{Z}\subseteq \mathbb{Z}[{1}/{d}]$. 
Let $\pi_H: \mathbb{C}\times (\mathbb{C}\setminus \bar{\mathbb{D}}) \rightarrow I_H^+$ and $\pi_F: \mathbb{C}\times (\mathbb{C}\setminus \bar{\mathbb{D}}) \rightarrow I_F^+$ be the covering maps. Now note that a biholomorphism $a$ from $I_H^+$ to $I_F^+$ can be lifted as an automorphism  $A$ of $\mathbb{C}\times (\mathbb{C}\setminus \bar{\mathbb{D}})$ if and only if the induced group isomorphism
\[
\pi_1(a): \mathbb{Z}\left[{1}/{d}\right] \rightarrow \mathbb{Z}\left[{1}/{d}\right]
\]
is $\pm\rm{Id}$ (identity maps). It is easy to see that any group isomorphism of $\mathbb{Z}[{1}/{d}]$ is of the form $x\mapsto \pm d^{s} x$, for some $s\in \mathbb{Z}$. Thus, there exists $n\in \mathbb{N}$ such that $F^{\pm n}\circ a$ induces identity on $\mathbb{Z}[{1}/{d}]$ and thus lifts as an automorphism of $\mathbb{C} \times \left(\mathbb{C}\setminus \bar{\mathbb{D}}\right)$. Therefore, without loss of generality, we assume that $a$ lifts as an automorphism $A$ of 
$\mathbb{C} \times \left(\mathbb{C}\setminus \bar{\mathbb{D}}\right)$.

\[ \begin{tikzcd}
\mathbb{C} \times \left(\mathbb{C}\setminus \bar{\mathbb{D}}\right) \arrow{r}{A} \arrow[swap]{d}{\pi_H} & \mathbb{C} \times \left(\mathbb{C}\setminus \bar{\mathbb{D}}\right) \arrow{d}{\pi_F} \\%
I_H^+ \arrow{r}{a}& I_F^+
\end{tikzcd}
\]
It follows from (\ref{fiber form}) that for $H$ there exists a polynomial $Q_H$ of degree $d+1$ such that for each element $\left[{k}/{d^n}\right] \in \mathbb{Z}[{1}/{d}]/\mathbb{Z}$, with $n\geq 0$ and $k\geq 1$, there exists a unique deck  transformation $\gamma_{{k}/{d^n}}$ from  $\mathbb{C}\times (\mathbb{C}\setminus\bar{\mathbb{D}})$ to $\mathbb{C}\times (\mathbb{C}\setminus\bar{\mathbb{D}})$ of the form
\begin{equation}\label{form aut}
\gamma_{{k}/{d^n}}
\begin{bmatrix}
z\\ \zeta
\end{bmatrix}
=\begin{bmatrix} 
 z+ \frac{d}{\delta} \sum_{l=0}^{n-1} {\left(\frac{d}{\delta}\right)}^l\left(Q_H(\zeta^{d^l})-Q_H\left( {\left(e^{\frac{2 k\pi i}{d^n}}\zeta)\right)}^{d^l} \right)\right)\\
 e^{\frac{2 k\pi i}{d^n}}\zeta
 \end{bmatrix}.
\end{equation}
 The same holds for the H\'{e}non map $F$. Thus if $\tilde{p}=(z,\zeta)\in \mathbb{C}\times \left(\mathbb{C}\setminus \bar{\mathbb{D}}\right)$ lies in the $\pi_H^{-1}(p)$ for some $p\in I_H^+$,  then
\[
\pi_H^{-1}(p)=\{\gamma_{{k}/{d^n}}(z,\zeta): n\geq 0, k\geq 1\}.
\]

\medskip 
\no 
{\it Form of lifts:}
First note that any automorphism  $A$ of $\mathbb{C}\times (\mathbb{C}\setminus \bar{\mathbb{D}})$ is of the form: 
\begin{equation}\label{aut cover}
A(z,\zeta) = (A_1(z,\zeta), A_2(z,\zeta))= (\beta (\zeta) z+ \gamma(\zeta), \alpha \zeta),
\end{equation}
where $\lvert \alpha \rvert=1$ and $\beta, \gamma$ are holomorphic maps from $\mathbb{C}\setminus \bar{\mathbb{D}}$ to $\mathbb{C}^*$ and $\mathbb{C}$, respectively (see \cite[Section 3]{bousch}). 

Note that since $a\circ \pi_H= \pi_F \circ A$,  if  $(z,\zeta), (z',\zeta') \in \pi_H^{-1}(p)$, for some $p\in I_H^+$, then $A(z,\zeta), A(z',\zeta') \in \pi_F^{-1}(a(p))$.  Now let $(z,\zeta)$ and $(z',\zeta')$ be in the same fiber of $\pi_H$,  then using (\ref{form aut}), we have
$${(\zeta'/\zeta)}^{d^n}=1$$
 and 
\[
z'=z+\frac{d}{\delta} \sum_{l=0}^{n-1} {\left(\frac{d}{\delta}\right)}^l\left(Q_H(\zeta^{d^l})-Q_H\left( {\left(e^{\frac{2 k\pi i}{d^n}}\zeta)\right)}^{d^l} \right)\right),
\]
for some $n\in \mathbb{N}$. Therefore the difference between the first coordinates of $A$, i.e.,
\begin{align}\label{A2uv}
A_1(z',\zeta')-A_1(z,\zeta)= 
\left(\beta(\zeta')-\beta(\zeta)\right)u+\beta(\zeta')\frac{d}{\delta} \sum_{l=0}^{n-1} {\left(\frac{d}{\delta}\right)}^l\left(Q_H(\zeta^{d^l})-Q_H\left( {\left(e^{\frac{2 k\pi i}{d^n}}\zeta)\right)}^{d^l} \right)\right)\nonumber \\
+\left(\gamma(\zeta')-\gamma(\zeta)\right).
\end{align}
Now since $A(z,\zeta)$ and $A(z',\zeta')$ are in the same fiber of $\pi_F$,  the difference 
$A_1(z,\zeta)-A_1(z',\zeta')$ is a function of $\alpha \zeta$ and $\alpha \zeta'$. Thus it follows from (\ref{A2uv}) that $\beta(\zeta)=\beta(\zeta')$, i.e.,
\[
\beta(\zeta)=\beta \left(\zeta. e^\frac{2\pi i k}{d^n}\right),
\]
for all $k\geq 1$ and for all $n\geq 0$. Therefore, $\beta(\zeta)\equiv \beta$ in $\mathbb{C}$. 
Thus it follows form (\ref{A2uv}) that
\begin{align}\label{alpha u}
A_1(z',\zeta')-A_1(z,\zeta)=\Delta_H(\zeta,\zeta')+\gamma(\zeta')-\gamma(\zeta),
\end{align}
where 
\[
\Delta_H(\zeta,\zeta')=\beta\frac{d}{\delta} \sum_{l=0}^{n-1} {\left(\frac{d}{\delta}\right)}^l\left(Q_H(\zeta^{d^l})-Q_H\left( {\left(e^{{2 k\pi i}/{d^n}}\zeta)\right)}^{d^l} \right)\right).
\]
On the other hand since $A(z,\zeta)$ and $A(z',\zeta')$ are in the same fiber of $\pi_F$ , 
\begin{equation}\label{alpha uv} 
A_1(z',\zeta')-A_1(z,\zeta)=\Delta_F(\alpha \zeta,\alpha \zeta')= \frac{d}{\delta} \sum_{l=0}^{n-1} {\left(\frac{d}{\delta}\right)}^l\left(Q_F(\alpha^{d^l} \zeta^{d^l})-Q_F\left( {\left(e^{\frac{2 k\pi i}{d^n}}\alpha \zeta)\right)}^{d^l} \right)\right). 
\end{equation}
Since $\gamma$ is a holomorphic function on $\mathbb{C}\setminus \bar{\mathbb{D}}$, comparing (\ref{alpha u}) and (\ref{alpha uv}), it follows that for a fixed $\zeta\in \mathbb{C}\setminus \bar{\mathbb{D}}$, the modulus of the difference between $\Delta_H(\zeta,\zeta')$ and $\Delta_F(\alpha \zeta,\alpha \zeta')$ is uniformly bounded for all $v'=e^{{2 k\pi i}/{d^n}}\zeta$ with $n\geq 0$ and $k\geq 1$. Note that for $\zeta'=e^{{2 k\pi i}/{d^n}}\zeta$, the difference  $\Delta_H(\zeta,\zeta')-\Delta_F(\alpha \zeta,\alpha \zeta')$ is a polynomial of degree $(d+1)d^{n-1}$ and it can be written as
\begin{equation}
\begin{aligned}\label{QHQF}
& \beta \sum_{l=0}^{n-1} {\left(\frac{d}{\delta}\right)}^l\left(Q_H\left(\zeta^{d^l}\right)-Q_H\left( {\left(e^{\frac{2 k\pi i}{d^n}}\zeta\right)}^{d^l} \right)\right)- \sum_{l=0}^{n-1} {\left(\frac{d}{\delta}\right)}^l\left(Q_F\left(\alpha^{d^l} \zeta^{d^l}\right)-Q_F\left( {\left(e^{\frac{2 k\pi i}{d^n}}\alpha \zeta\right)}^{d^l} \right)\right) \nonumber \\
=& \sum_{l=0}^{n-1} {\left(\frac{d}{\delta}\right)}^l\left[\beta Q_H\left(\zeta^{d^l}\right)-Q_F\left(\alpha^{d^l} \zeta^{d^l}\right)\right]- \sum_{l=0}^{n-1} {\left(\frac{d}{\delta}\right)}^l\left[\beta Q_H\left( {\left(e^{\frac{2 k\pi i}{d^n}}\zeta\right)}^{d^l} \right)- Q_F\left({\left(e^{\frac{2 k\pi i}{d^n}}\alpha \zeta\right)}^{d^l} \right)\right] \nonumber\\
=&\frac{d^{n-1}}{\delta^{n-1}} \left[ \beta {\zeta}^{d^{n-1}(d+1)}-{\left(\alpha \zeta \right)}^{d^{n-1} (d+1)}-
 \beta {\left(e^{\frac{2 \pi i}{d^n}}\zeta\right)}^{d^{n-1}(d+1)}+{\left(\alpha  \zeta  e^{\frac{2 \pi i}{d^n}}\right)}^{d^{n-1}(d+1)}\right] +R(\zeta) \nonumber \\ 
 =& {\left(\frac{d}{\delta}\right)}^{n-1} \left( \beta- \alpha^{d^{n-1}(d+1)}\right)\left( 1-e^{\frac{2\pi i}{d}}\right)\zeta^{d^{n-1}(d+1)} \left[1 + \tilde{R}(\zeta)\right],
\end{aligned}
\end{equation}
where $R(\zeta)=O\left(\zeta^{d^n}\right)$  and $$\tilde{R}(\zeta)=R(\zeta){\left({d}/{\delta}\right)}^{-n+1} {\left( \beta- \alpha^{d^{n-1}(d+1)}\right)}^{-1}{\left( 1-e^{\frac{2\pi i}{d}}\right)}^{-1}\zeta^{-d^{n-1}(d+1)}.$$
We claim that $\alpha^{d^{n-1}(d+1)}\rightarrow \beta$, as $n\rightarrow \infty$.
If not, then there exists a subsequence $\{n_l\}_{l\geq 1}$ such that 
\begin{equation}\label{alphaa beta}
\left\lvert \alpha^{d^{n_l-1}(d+1)}- \beta\right\rvert >c>0,
\end{equation}
 for all $l\geq 1$. Thus, if $\lvert {d}/{\delta} \rvert\geq1$, then 
(note that $R$ is a polynomial  of degree at most $d^n$) 
\begin{equation*} \label{alpha beta}
\lvert \tilde{R}(\zeta)\rvert \leq \frac{4nK_1}{{\lvert \zeta \rvert}^{d^{n-1}}},
\end{equation*} 
for some $K_1>1$. On the other hand, if $\lvert {d}/{\delta} \rvert <1$, then 
\begin{equation*} \label{alpha beta}
\lvert \tilde{R}(\zeta)\rvert \leq \frac{4nK_2}{{\lvert \zeta \rvert}^{d^{n-1}}}{\left(\frac{\delta}{d}\right)}^{n-1},
\end{equation*} 
for some $K_2>1$.  Therefore, since the modulus of $\Delta_H\left(\zeta,e^{\frac{2\pi i}{d^n}}\zeta\right)-\Delta_F\left(\alpha \zeta, \alpha e^{\frac{2\pi i}{d^n}}\zeta\right)$ is uniformly bounded, for all $n\geq 1$, we get a contradiction if (\ref{alphaa beta}) holds. 
Thus
\begin{equation}\label{beta one}
\alpha^{d^{n}(d+1)} \rightarrow  \beta,
\end{equation}
as $n \rightarrow \infty$.  Also, 
\begin{equation}\label{beta two}
\alpha^{d^{n+1}(d+1)} \rightarrow  \beta
\end{equation}
as $n\rightarrow \infty$ and thus dividing  (\ref{beta two}) by  (\ref{beta one}) and taking the limit, we get 
\begin{equation}
\alpha^{(d+1)(d-1)d^{n}} \rightarrow  1,
\end{equation}
as  $n\rightarrow \infty$. Therefore, we get 
\begin{equation*}\label{beta}
\beta^{d-1}=1.
\end{equation*} 
On the other hand, it follows from (\ref{beta one}) that $ {\left(\alpha^{d+1}\right)}^{d^n} \rightarrow \beta$ and also note that $\beta$ is a repelling fixed point for the map $z\mapsto z^d$. Therefore, (\ref{beta one}) holds if and only if
\begin{equation*}\label{alpha d}
\alpha^{d+1}=\beta.
\end{equation*}
Thus $\left(\alpha^{d^{n}(d+1)} - \beta\right)=0$, for all $n\geq 1$, which in turn gives $\gamma(\zeta)=\gamma\left(e^{\frac{2k\pi i}{d^n}}\zeta\right)$, for all $k\ge1$ and $n\geq 0$. Thus $\gamma\equiv \gamma_0$, for some $\gamma_0\in \mathbb{C}$.

\medskip
\no 
{\it Relation between $Q_H$ and $Q_F$:} Note that it follows from Lemma \ref{yzetaQ} that next to the highest degree coefficients of $Q_H$ and $Q_F$ vanish. Let 
 \begin{equation}\label{QH}
 Q_H(\zeta)=\zeta^{d+1}+A^H_{d-1} \zeta^{d-1}+ A_{d-2}^H \zeta^{d-2} + \cdots + A_1^H \zeta+A_0^H
 \end{equation}
 and 
 \begin{equation}\label{QF}
 Q_F(\zeta)=\zeta^{d+1}+A^F_{d-1} \zeta^{d-1}+ A_{d-2}^F \zeta^{d-2} + \cdots + A_1^F \zeta+A_0^F.
 \end{equation}
  Since $\gamma$ is identically constant in the complex plane, taking $k,n=1$, it follows from
 (\ref{alpha u}) and  (\ref{alpha uv}) that 
 \begin{equation}\label{ReQHQF}
 \beta \frac{d}{\delta} \left[Q_H(\zeta)-Q_H\left(e^{\frac{2\pi i}{d}} \zeta\right)\right]=\Delta_H(\zeta, \zeta')=\Delta_F (\alpha \zeta, \alpha \zeta')=\frac{d}{\delta} \left[Q_F(\alpha \zeta)-Q_F\left(e^{\frac{2\pi i}{d}} \alpha \zeta\right)\right].
 \end{equation}
 Comparing coefficients of both sides of (\ref{ReQHQF}), we get
 \begin{equation*}
 \beta A_{d-k}^H=\alpha^{d-k} A_{d-k}^F,
 \end{equation*}
 for $1\leq k \leq (d-1)$ and since $\beta=\alpha^{d+1}$, equivalently we get
  \begin{equation}\label{coeffQ}
  A_{d-k}^H=\alpha^{-(k+1)} A_{d-k}^F,
 \end{equation}
 for $1\leq k \leq (d-1)$. Note that the relation between the constant terms of $Q_H$ and $Q_F$, i.e., the relation between $A_0^H$ and $A_0^F$ cannot be extracted from (\ref{ReQHQF}). However, as we are going to see in the next section that the explicit relation between $A_{d-i}^H$ and  $A_{d-i}^F$, for $1\leq i\leq (d-1)$, obtained in (\ref{coeffQ}) is sufficient to track down the relation between $p_H$ and $p_F$.

\section 
{Relation between $p_H$ and $p_F$}
\no 
It follows from Lemma \ref{yzetaQ} that there exist 
\begin{equation} \label{yH}
 y_H(\zeta)\equiv y_H(0, \zeta) =\zeta+\frac{D_1^H}{\zeta}+ \frac{D_2^H}{\zeta^2}
 +\cdots+ \frac{D_{d-1}^H}{\zeta^{d-1}}
   +O\left(\frac{1}{\zeta^d}\right)
 \end{equation}
 and 
 \begin{equation}\label{yF}
 y_F(\zeta)\equiv y_F(0, \zeta)=\zeta+ \frac{D_1^F}{\zeta}+ \frac{D_2^F}{\zeta^2}
 +\cdots+ \frac{D_{d-1}^F}{\zeta^{d-1}} + O\left(\frac{1}{\zeta^d}\right),
 \end{equation}
 with $D_i^H, D_i^F\in \mathbb{C}$ for $i\geq 1$, such that $Q_H(\zeta)$ and $Q_F(\zeta)$ are polynomial parts of $\zeta^d y_H(0,\zeta)$ and $\zeta^d y_F(0,\zeta)$, respectively.  Therefore, the precise relation obtained in (\ref{coeffQ}) determines a similar relation between the first few corresponding coefficients of the power series expansion of $y_H(\zeta) $ and $y_F(\zeta)$, namely
 \begin{equation} \label{DHDF}
 D_k^H=\alpha^{-(k+1)}D_k^F,
 \end{equation}
  for $1\leq k \leq (d-1)$. 
  
  Since we assume that the Jacobian determinants of $H$ and $F$ are the same, that is, $\delta_H=\delta_F=\delta$,  understanding the relation between the polynomials $p_H$ and $p_F$ is sufficient to capture the relation between $H$ and $F$.  We  show in this section  that the relation between the coefficients of $p_H$ and $p_F$ can be extracted via the B\"{o}ttcher coordinates of $H$ and $F$, namely via the functions $\zeta_H$ and $\zeta_F$, respectively. Let 
  \begin{equation}\label{zetaH}
 \zeta_H(y)\equiv  \zeta_H(0,y)=y+\frac{L_1^H}{y}+ \frac{L_2^H}{y^2}
 +\cdots+ \frac{L_{d-1}^H}{y^{d-1}} + O\left(\frac{1}{y^d}\right)
 \end{equation}
 and 
 \begin{equation}\label{zetaF}
 \zeta_F(y)\equiv  \zeta_F(0,y)=y+\frac{L_1^F}{y}+ \frac{L_2^F}{y^2}
 +\cdots+ \frac{L_{d-1}^F}{y^{d-1}} + O\left(\frac{1}{y^d}\right).
 \end{equation}
 Recall from Section 2 that $\zeta_H \circ y_H (0,\zeta)= \zeta$ and $\zeta_F \circ y_F (0,\zeta)= \zeta$, for all $\zeta\in \mathbb{C}$ with $\lvert \zeta\rvert>R$, where $R>0$ is sufficiently large. Thus implementing (\ref{yH}),  (\ref{yF}), (\ref{zetaH}) and (\ref{zetaF}) together,  we get
\begin{align} \label{LH1}
 \left[\frac{D_1^H}{\zeta}+\frac{D_2^H}{\zeta^2}
 +\cdots+ \frac{D_{d-1}^H}{\zeta^{d-1}} + O\left(\frac{1}{\zeta^d}\right)\right]+\frac{L_1^H}{\zeta}  \left[1+\frac{D_1^H}{\zeta^2}+\frac{D_2^H}{\zeta^3}
 +\cdots+ \frac{D_{d-1}^H}{\zeta^{d-1}} + O\left(\frac{1}{\zeta^d}\right)\right]^{-1} \nonumber \\
 +\frac{L_2^H}{\zeta^2}  \left[1+\frac{D_1^H}{\zeta^2}+\frac{D_2^H}{\zeta^3}
 +\cdots+ \frac{D_{d-1}^H}{\zeta^{d-1}} + O\left(\frac{1}{\zeta^d}\right)\right]^{-2}+ \cdots=0
 \end{align}
 and 
 \begin{align} \label{LF1}
 \left[\frac{D_1^F}{\zeta}+\frac{D_2^F}{\zeta^2}
 +\cdots+ \frac{D_{d}^F}{\zeta^d} + O\left(\frac{1}{\zeta^d}\right)\right]+\frac{L_1^F}{\zeta}  \left[1+\frac{D_1^F}{\zeta^2}+\frac{D_2^F}{\zeta^3}
 +\cdots+ \frac{D_{d-1}^F}{\zeta^{d-1}} + O\left(\frac{1}{\zeta^d}\right)\right]^{-1} \nonumber \\
+ \frac{L_2^F}{\zeta^2}  \left[1+\frac{D_1^F}{\zeta^2}+\frac{D_2^F}{\zeta^3}
 +\cdots+ \frac{D_{d-1}^F}{\zeta^{d-1}} + O\left(\frac{1}{\zeta^d}\right)\right]^{-2}+ \cdots=0.
 \end{align}
Note that expanding (\ref{LH1}) and (\ref{LF1}),  one can express $L_k^H$ and $L_k^F$, for any $k\geq 1$, in terms of $D_1^H, D_2^H,\ldots$ and $D_1^F, D_2^F,\ldots$, respectively. In fact, using the relation between $D_k^H$ and $D_k^F$, for $1\leq k \leq (d-1)$, obtained in (\ref{DHDF}), one can establish an explicit relation between $L_k^H$ and $L_k^F$ for $1\leq k \leq (d-1)$.
 
\medskip 
\no 
{\it Claim} $\mathcal{L}$ : For $1\leq k \leq d-1$, $L_k^H=\alpha^{-(k+1)}L_k^F$. 

\medskip 
\no 
Note that implementing claim $\mathcal{L}$ to (\ref{zetaH}) and (\ref{zetaF}), we obtain
\begin{equation} \label{zeta}
\alpha\zeta_H(y)-\zeta_F(\alpha y)=O\left({1}/{y^{d}}\right).
\end{equation}
As we shall see shortly that the above relation between $\zeta_H$ and $\zeta_F$ plays the key role in  establishing the relation between $p_H$ and $p_F$. For now assuming claim $\mathcal{L}$ is true and thus assuming (\ref{zeta}) holds, let us first  determine the relation between $p_H$ and $p_F$.  We see a proof of the claim $\mathcal{L}$ in the end of the present section. 

\medskip 
\no 
{\it Relation between $p_H$ and $p_F$:}
Recall from Section 2 that for $\lvert y\rvert >R$, with $R$ large enough, one can write
\begin{equation}\label{twoterm}
\zeta_H(y)\equiv \zeta_H(0,y)=y.{\left(\frac{p_H(y)}{y^d}\right)}^{\frac{1}{d}}{\left(\frac{p_H(y_{1,H})}{y_{1,H}^d}\right)}^{\frac{1}{d^2}}\cdots
\end{equation}
and 
\begin{equation}\label{twoterm}
\zeta_F(y)\equiv \zeta_F(0,y)=y.{\left(\frac{p_F(y)}{y^d}\right)}^{\frac{1}{d}}{\left(\frac{p_F(y_{1,F})}{y_{1,F}^d}\right)}^{\frac{1}{d^2}}\cdots,
\end{equation}
where $y_{1,H}={(H(x,y))_2}=p_H(y)-\delta x$ and $y_{1,F}={(F(x,y))_2}=p_F(y)-\delta x$.
Thus  $\zeta_H(y)-O({1}/{y^{d}})$ and $\zeta_F(y)-O({1}/{y^{d}})$ are determined by
\[
 y.{\left(\frac{p_H(y)}{y^d}\right)}^{\frac{1}{d}} \text{ and } y.{\left(\frac{p_F(y)}{y^d}\right)}^{\frac{1}{d}},
\]
respectively. 
The power series expansion of the holomorphic function ${(1+z)}^{{1}/{d}}$ (principal branch) is
$$1+({1}/{d}) z+\left({({1}/{d}-1)}/{2d}\right) z^2+\cdots,$$
for  $z\in \mathbb{D}=\left\{z\in \mathbb{C}: \lvert z\rvert<1\right\}$. 
Now note that 
  \begin{eqnarray}\label{coeff pH}
   y.{\left(\frac{p_H(y)}{y^d}\right)}^{\frac{1}{d}}
  &=& y. {\left(\frac{y^d+a_{d-2}^H y^{d-2}+\cdots+a_1^H y+a_0^H}{y^d}\right)}^{\frac{1}{d}} \nonumber \\
 & =&y \left[1+ \frac{1}{d} \left( \frac{a_{d-2}^H}{y^2}+\cdots+\frac{a_0^H}{y^d}\right)+\frac{1}{2d}\left(\frac{1}{d}-1\right){\left( \frac{a_{d-2}^H}{y^2}+\cdots+\frac{a_0^H}{y^d}\right)}^2+\cdots\right]
 \end{eqnarray}
 and 
  \begin{eqnarray}\label{coeff pF}
  \alpha y.{\left(\frac{p_F(\alpha y)}{{(\alpha y)}^d}\right)}^{\frac{1}{d}}
  &=&\alpha  y. {\left(\frac{{(\alpha y)}^d+a_{d-2}^F {(\alpha y)}^{d-2}+\cdots+a_1^F (\alpha y)+a_0^F}{{(\alpha y)}^d}\right)}^{\frac{1}{d}} \nonumber \\
 & =&\alpha y \left[1+ \frac{1}{d}\left( \frac{a_{d-2}^F}{\alpha^2 y^2}+\cdots+\frac{a_0^F}{\alpha^d y^d}\right)+\frac{1}{2d}\left(\frac{1}{d}-1\right){\left( \frac{a_{d-2}^F}{\alpha^2 y^2}+\cdots+\frac{a_0^F}{\alpha^d y^d}\right)}^2+\cdots\right],
 \end{eqnarray}
 for all $y\in \mathbb{C}$  with $\lvert y \rvert$ large enough.
Since (\ref{zeta}) holds, comparing (\ref{coeff pH}) and  (\ref{coeff pF}), we get
$$
a_{d-2}^F=\alpha^2 a_{d-2}^H.
$$
Let us assume that 
$$a_{d-k}^F=\alpha^k a_{d-k}^H,$$
up to some $k$, where $2\leq k \leq (d-1)$. Now it follows from (\ref{coeff pH}) that the coefficient of ${1}/{y^{k}}$ in the expansion of $\alpha \zeta_H(y)$ is 
\begin{equation}\label{ckH}
c_{k+1}^H=\alpha\left[\left(a_{d-(k+1)}^H/{d}\right)+ G\left(a_{d-2}^H, \cdots,a_{d-k}^H \right)\right],
\end{equation}
where $G$ is defined as follows. The function $G$ is the polynomial in $(k-1)$ complex variables determined by the coefficient of ${1}/{y^{k+1}}$ of the power series expansion of 
\begin{equation}\label{def f}
{\left(\frac{p_H(y)}{y^d}\right)}^{\frac{1}{d}}-\frac{1}{d} \left( \frac{a_{d-2}^H}{y^2}+\cdots+\frac{a_0^H}{y^d}\right) -1=\frac{1}{2d}\left(\frac{1}{d}-1\right){\left( \frac{a_{d-2}^H}{ y^2}+\cdots+\frac{a_0^H}{ y^d}\right)}^2+\cdots
\end{equation}
 (see \ref{coeff pH}). In other words, if the coefficient of ${1}/{y^{k+1}}$ in (\ref{def f}) is $\sum_{i=1}^n g_i \left(\Pi_{j=2}^k {(a_{d-j}^H)}^{i(j)}\right)$ with $g_i\in \mathbb{C}$, then 
 $$G(x_1,x_2,\ldots,x_{k-1})=\sum_{i=1}^n g_i \left(\Pi_{j=2}^k {x_j}^{i(j)}\right).$$
Similarly, it follows from (\ref{coeff pF}) that the coefficient of ${1}/{y^{k}}$ in the expansion of $ \zeta_F(\alpha y)$ is 
 \begin{equation}\label{ckF}
  c_{k+1}^F=\frac{a_{d-(k+1)}^F}{\alpha^{k} d}+\alpha G\left (\frac{a_{d-2}^F}{\alpha^2}, \cdots,\frac{a_{d-k}^F}{\alpha^k} \right)=\frac{a_{d-(k+1)}^F}{\alpha^{k} d}+ \alpha G(a_{d-2}^H, \cdots,a_{d-k}^H ).
 \end{equation}
 Since $c_{k+1}^H=c_{k+1}^F$, comparing (\ref{ckH}) and (\ref{ckF}), we have 
 \[
 a_{d-(k+1)}^F=\alpha^{k+1}a_{d-(k+1)}^H.
 \]
 Therefore, the following relation between the corresponding coefficients of $p_H$ and $p_F$ holds:
 \begin{equation}\label{rel p coeff}
  a_{d-k}^F=\alpha^{k}a_{d-k}^H,
 \end{equation}
 for $2\leq k \leq d$.
 
 Now since $\beta=\alpha^{d+1}$ and (\ref{rel p coeff}) holds, we have
 \begin{eqnarray*}
 p_F(y)&=&y^d+a_{d-2}^F y^{d-2}+\cdots+a_1^F y+a_0^F\\
 &=&y^d+ \alpha^2 a_{d-2}^H y^{d-2}+\alpha^3 a_{d-3}^H y^{d-3}+\cdots+\alpha^{d-1} a_{1}^H y+\alpha^d a_0^H\\
 &=& y^d+\beta \alpha^{-(d-1)} a_{d-2}^H y^{d-2}+\beta \alpha^{-(d-2)} a_{d-3}^H y^{d-3}+\cdots+\beta \alpha^{-2} a_{1}^H y+\beta \alpha^{-1} a_0^H\\
& =&\beta \alpha^{-1} \left({(\alpha^{-1}y)}^d+ a_{d-2}^H {(\alpha^{-1}y)}^{d-2}+ a_{d-3}^H {(\alpha^{-1}y)}^{d-3}+\cdots+ a_{1}^H {(\alpha^{-1}y)}+a_0^H\right)\\
&=& (\beta \alpha^{-1}) p_H(\alpha^{-1}y).
 \end{eqnarray*}
 In other words, we obtain 
 \begin{equation} \label{rel phpf}
 \alpha p_F(\alpha y)=\beta p_H(y),
 \end{equation}
 for all $y\in \mathbb{C}$, with $\alpha^{d+1}=\beta$ and $\beta^{d-1}=1$.

\medskip 
\no 
Before giving a formal proof of the claim $\mathcal{L}$, let us establish the relation between $L_k^H$ and $L_k^F$ with bare hands when the common degree $d$ of the H\'{e}non maps $H$ and $F$ is small. To start with, first note that 
 the coefficients of ${1}/{\zeta^k}$ vanish in both (\ref{LH1}) and (\ref{LF1}), for all $k\geq 1$.

\medskip 
 \no 
 $\bullet$ Let $d=2$. Thus by (\ref{DHDF}), we have $D_1^H=\alpha^{-2}D_1^F $. The coefficients of ${1}/{\zeta}$ in (\ref{LH1}) and (\ref{LF1}) are $D_1^H+ L_1^H$ and $D_1^F+ L_1^F$, respectively and both vanish. Thus
 \[
  L_1^H=-D_1^H=-\alpha^{-2} D_1^F=\alpha^{-2} L_1^F.
 \]

 \medskip 
 \no 
 $\bullet$ Let $d=3$. By (\ref{DHDF}), we have $D_i^H=\alpha^{-(i+1)}D_i^F $, for $i=1,2$.
 Since the coefficients of ${1}/{\zeta}$ vanish in (\ref{LH1}) and (\ref{LF1}), as before we can show that $L_1^H=\alpha^{-2} L_1^F$.  Using the fact that coefficients of ${1}/{\zeta^2}$, which are 
 $$D_2^H+L_1^H. 0+L_2^H \text{ and } D_2^F+L_1^F. 0+L_2^F, $$  
 vanish,  we have 
 \[
  L_2^H=-D_2^H=-\alpha^{-3} D_2^F= \alpha^{-3}   L_2^F.
 \]
 
 \medskip 
 \no 
 $\bullet$ Let $d=4$. By (\ref{DHDF}), we have $D_i^H=\alpha^{-(i+1)}D_i^F$, for $1\leq i \leq 3$. Now since the coefficients ${1}/{\zeta^i}$ vanish in (\ref{LH1}) and (\ref{LF1}), as before we can show that $L_i^H=\alpha^{-(i+1)}  L_i^F$, for $1\leq i \leq 2$. Now the coefficients of ${1}/{\zeta^3}$ in  (\ref{LH1}) and (\ref{LF1}) are 
 \[
D_3^H+L_1^H. \left(C_{1}D_1^H \right)+L_2^H. 0+ L_3^H \text{ and } D_3^F+L_1^F. \left(C_{1}D_1^F \right)+L_2^F. 0+ L_3^F,
\]
respectively for some $C_{1}\in \mathbb{C}$ and they vanish. Therefore, 
 \begin{eqnarray*}
&& L_3^H =-D_3^H-L_1^H. \left(C_{1}D_1^H\right), 
\end{eqnarray*}
which implies 
\[
L_3^H =-\alpha^{-4} D_3^F-\alpha^{-4}  L_1^F . \left(C_{1} D_1^F\right)= \alpha^{-4} L_3^F.
\]
 
 \medskip 
 \no 
 $\bullet$ Let $d=5$. Thus we have $D_i^H=\alpha^{-(i+1)}D_i^F$, for $1\leq i \leq 4$. Using the same arguments as before we can show $L_i^H=\alpha^{-(i+1)}  L_i^F$, for $1\leq i \leq 3$. Now the coefficients of ${1}/{\zeta^4}$ in  (\ref{LH1}) and (\ref{LF1})  are 
 \[
 D_4^H+ L_1^H. \left(C_{2} D_2^H\right)+L_2^H. \left(C_{3} D_1^H\right)+ L_3^H. 0
 + L_4^H
 \]
 and 
 \[
 D_4^F+ L_1^F. \left(C_{2} D_2^F\right)+L_2^F. \left(C_{3} D_1^F\right)+ L_3^F. 0
 + L_4^F,
 \]
 for some $C_2, C_3 \in \mathbb{C}$ and since they vanish, a simple calculation as above gives that 
 \[
 L_4^H= \alpha^{-5} L_4^F.
 \]
 
 \medskip 
 \no 
$\bullet$ Let $d=6$. Thus we have $D_i^H=\alpha^{-(i+1)}D_i^F$, for $1\leq i \leq 5$. As before we can show $L_i^H=\alpha^{-(i+1)}  L_i^F$, for $1\leq i \leq 4$. The coefficients of ${1}/{\zeta^5}$ in  (\ref{LH1}) and (\ref{LF1})  are $M_5^X$. Here for $X=H, F$,
 \begin{align*}
M_5^X= D_5^X + L_1^X. \left (C_4 D_3^X+C_5{(D_1^X)}^2\right)+L_2^X. \left(C_6 D_2^X\right)+ L_3^X . (C_7 D_1^X)+L_4^X .0+ L_5^X,
 \end{align*}
 where $C_4, C_5$ and so on are from $\mathbb{C}$. Since $M_5^H$ and $M_5^F$ vanish, 
 calculations as above give us 
 \[
 L_5^H= \alpha^{-6} L_5^F.
 \]
 
 \medskip 
 \no 
$\bullet$ Let $d=7$. Thus $D_i^H=\alpha^{-(i+1)}D_i^F$, for $1\leq i \leq 6$. Hence we can show $L_i^H=\alpha^{-(i+1)}  L_i^F$, for $1\leq i \leq 5$ as before.  The coefficients of ${1}/{\zeta^6}$ in  (\ref{LH1}) and (\ref{LF1})  are $M_6^X$. Here for $X=H, F$, 
 \begin{align*}
 M_6^X=D_6^X + L_1^X. \left (C_8 D_4^X+C_9 D_1^X D_2^X\right)+L_2^X.\left(C_{10} D_3^X+C_{11} {(D_1^X)}^2\right)+ L_3^X. (C_{12}D_2^X)\\
 +L_4^X. (C_{13} D_1^X)
 + L_5^X . 0+L_6^X,
 \end{align*}
 where $C_8, C_9$ and so on are from $\mathbb{C}$.
 Since $M_6^H$ and $M_6^F$ vanish,  we get
 \[
 L_6^H= \alpha^{-7} L_6^F.
 \]
 
 \medskip 
 \no 
 $\bullet$ Let $d=8$. So $D_i^H=\alpha^{-(i+1)}D_i^F$, for $1\leq i \leq 7$. Hence we can show $L_i^H=\alpha^{-(i+1)}  L_i^F$, for $1\leq i \leq 6$ just as before. 
 Now the the coefficients of ${1}/{\zeta^7}$ in (\ref{LH1}) and (\ref{LF1})  are $M_7^X$. Here for $X=H, F$, 
 \begin{eqnarray*}
M_7^X&=& D_7^X + L_1^X.\left(C_{14} D_5^X+C_{15} D_1^X D_3^X+ C_{16}{(D_2^X)}^2+ C_{17} {(D_1^X)}^3 \right)\\
&+& L_2^X.\left (C_{18} D_4^X+C_{19} D_1^X D_2^X\right)
 + L_3^X.\left(C_{20} D_3^X +C_{21} {(D_1^X)}^2\right)\\
 &+&L_4^X .(C_{22} D_2^X)+ L_5^X .(C_{23} D_1^X)+ L_6^X.0+L_7^X,
 \end{eqnarray*}
 where $C_{14}, C_{15}$ and so on are from $\mathbb{C}$.
 A simple calculation as above gives us 
 \[
 L_7^H= \alpha^{-8} L_7^F.
 \]
 As promised earlier, now we see a proof of the claim $\mathcal{L}$. 
\subsection*{Proof of Claim $\mathcal{L}$}
First we make a few observations, which follow immediately by chasing the expression of the coefficients of ${1}/{\zeta^s}$ in (\ref{LH1}) and (\ref{LF1}), for $s\geq 1$ and then using the fact that all of them vanish. 

\medskip 
\no 
{\textbf{(O1)}}
For any $s\geq 1$, the coefficients of ${1}/{\zeta^s}$ in (\ref{LH1}) and (\ref{LF1}) vanish. Further,  they are of the form $D_s^H+R_s^H+ L_s^H$ and $D_s^F+R_s^F+ L_s^F$, respectively, where $R_s^H$ and $R_s^F$ are linear combinations of products of powers of $D_i^H$'s, $L_j^H$'s and $D_i^F$'s, $L_j^F$'s, respectively, with $1\leq i,j<s$.  

\medskip 
\no 
 {\textbf{(O2)}}
 For $s \geq 1$, one can write
 \[
R_s^H= \sum_{i=1}^{s-1} L_i^H R_{i,s}^H \text{ and } R_s^F=\sum_{i=1}^{s-1} L_i^F R_{i,s}^F.
 \]
Let $2\leq i \leq s-1$. It follows from (\ref{LH1}) that $R_{i,s}^H$ is the coefficient of ${1}/{\zeta^{s-i}}$ of the power series expansion of
  \[
  \left(1+\frac{D_1^H}{\zeta^2}+\frac{D_2^H}{\zeta^3}
 +\cdots+ \frac{D_{d-1}^H}{\zeta^{d-1}} + O\left(\frac{1}{\zeta^d}\right)\right)^{-i},
  \]
 and  similarly, 
  $R_{i-1,s-1}^H$ is the coefficient of ${1}/{\zeta^{s-i}}$ of the power series expansion of
  \[
  \left(1+\frac{D_1^H}{\zeta^2}+\frac{D_2^H}{\zeta^3}
 +\cdots+ \frac{D_{d-1}^H}{\zeta^{d-1}}+ O\left(\frac{1}{\zeta^d}\right)\right)^{-(i-1)}.
 \]
Therefore, for $2\leq i \leq s-1$, if
 \begin{equation} \label{Ris}
 R_{i,s}^H=C(i,s,1) D^H(i,s,1) +C(i,s,2) D^H(i,s,2)+\cdots +C(i,s,i_s) D^H(i,s,i_s)
 \end{equation}
 and 
 \begin{equation} 
  R_{i,s}^F=C(i,s,1) D^F(i,s,1) +C(i,s,2) D^F(i,s,2)+\cdots +C(i,s,i_s) D^F(i,s,i_s),
  \end{equation}
  where for $1\leq j \leq i_s$, $D^H(i,s,j)$ and $D^F(i,s,j)$ are products of powers of $D_l^H$'s and $D_l^F$'s ($1\leq l <s$), respectively,  with $ C(i,s,j)\in \mathbb{C}$, then 
\begin{equation}\label{Riss}
R_{i-1,s-1}^H= \tilde{C}(i,s,1)D^H(i,s,1) + \tilde{C}(i,s,2)D^H(i,s,2) + \cdots +\tilde{C}{(i,s,i_s)} D^H(i,s,i_s) 
 \end{equation}
 and 
 \begin{equation}
R_{i-1,s-1}^F= \tilde{C}(i,s,1)D^F(i,s,1) + \tilde{C}(i,s,2)D^F(i,s,2) + \cdots +\tilde{C}{(i,s,i_s)} D^F(i,s,i_s),
 \end{equation}
 with $\tilde{C}(i,s,j)\in \mathbb{C}$, for $1\leq j \leq i_s$. 
 
 \medskip 
 \no 
{\textbf{(O3)}}
Let $ s \geq 1$ and $i=1$. Then 
\begin{equation} 
 R_{1,s}^H=C(1,s,1) D^H(i,s,1) +C(1,s,2) D^H(i,s,2)+\cdots +C(1,s,1_s) D^H(1,s,1_s),
 \end{equation}
where for $1\leq j \leq 1_s$, 
\[
D^H(1,s,j)=D_{{s_j}_1}^H D_{{s_j}_2}^H \cdots D_{{s_j}_{m(s_j)}}^H,
\]
with ${s_j}_1,{s_j}_2,\ldots, {s_j}_{m(s_j)} \in \{1,2,\ldots,s-1\}$ (${s_j}_1,{s_j}_2,\ldots, {s_j}_{m(s_j)}$ are not possibly all distinct), i.e., $D_{{s_j}_1}^H, D_{{s_j}_2}^H, \ldots, D_{{s_j}_{m(s_j)}}^H \in \left\{D_1^H, \ldots, D_{s-1}^H\right\}$. Now $R_{1,s}^H$ is the coefficient of ${1}/{\zeta^{s-1}}$ of the power series expansion of
  \[
  \left(1+\frac{D_1^H}{\zeta^2}+\frac{D_2^H}{\zeta^3}
 +\cdots+ \frac{D_{d-1}^H}{\zeta^{d-1}}+ O\left(\frac{1}{\zeta^d}\right)\right)^{-1}.
 \]
 Similarly, $R_{1,s+1}^H$ is the coefficient of ${1}/{\zeta^{s}}$ of the power series expansion of
  \[
  \left(1+\frac{D_1^H}{\zeta^2}+\frac{D_2^H}{\zeta^3}
 +\cdots+ \frac{D_{d-1}^H}{\zeta^{d-1}}+ O\left(\frac{1}{\zeta^d}\right)\right)^{-1}.
 \]
Therefore, if 
  \begin{equation} \label{R1sH}
  R_{1,s}^H= \sum_{j=1}^{1_s} C(1,s,j) D_{{s_j}_1}^H D_{{s_j}_2}^H \cdots D_{{s_j}_{m(s_j)}}^H,
\end{equation}
then 
\begin{equation}\label{R1s1H}
 \begin{aligned}
 R_{1,s+1}^H =&\sum_{j=1}^{1_s} \biggl[C^{(1)}{(1,s,j)} D_{{s_j}_1+1}^H D_{{s_j}_2} \cdots D_{{s_j}_{m(s_j)}}^H +C^{(2)}{(1,s,j)}D_{{s_j}_1}^H D_{{s_j}_2+1}^H \cdots D_{{s_j}_{m(s_j)}}^H\\
&  +\cdots +C^{(m(s_j))}(1,s,j) D_{{s_j}_1}^H.D_{{s_j}_2}\cdots D_{{s_j}_{m(s_j)}+1}^H\biggr]+C_s {(D_1^H)}^{\frac{s}{2}} , 
   \end{aligned}
   \end{equation}
   with $C_s, C^{(k)}(1,s,j)\in \mathbb{C}$, for $1\leq j \leq 1_s$ and $1\leq k \leq m(s_j)$. 
  Further,  $C_s=0$, if $s$ is not divisible by $2$ and $C_s\neq 0$, if $s$ is divisible by $2$. The same relation holds between $R_{1,s}^F$ and $R_{1,s+1}^F$.

\medskip 
 \no 
 
 Let $d\geq 4$. Note that we have already established claim $\mathcal{L}$ for $d=2,3$ just before starting the present proof. Let $1\leq k \leq d-2$. Then we prove that if for $1\leq s \leq k$,
\begin{itemize}
\item[$(\bullet)$]
$L_{s}^H=\alpha^{-(s+1)}L_s^F$, and

\item[$(\bullet \bullet)$] 
for each fixed pair $(i,s)$ satisfying $2 \leq i\leq s-1$,  
$
D^H(i,s,j)= \alpha^{-(s-i)}D^F(i,s,j),
$
for $1\leq j \leq i_s$ and $R_{1,s}^H=\alpha^{-(s-1)}R_{1,s}^F$ (which in turn gives $R_{i,s}^H=\alpha^{-(s-i)}R_{i,s}^F$, for  $1\leq i \leq s-1$),
\end{itemize}
 then $L_{k+1}^H=\alpha^{-(k+2)}L_{k+1}^F$.
  
Now observe that performing the same calculations, which we did immediately before starting this proof, ($\bullet$) and ($\bullet \bullet$) hold for a first few $k$'s with $1\leq k \leq d-2$.  The coefficients of ${1}/{\zeta^{k+1}}$ in (\ref{LH1}) and (\ref{LF1}) are
 \begin{equation}\label{vanishH}
 D_{k+1}^H + L_1^H  R_{1,k+1}^H + L_2^H R_{2,k+1}^H + \cdots + L_{k}^H R_{k,k+1}^H+  L_{k+1}^H,
 \end{equation}
 and 
  \begin{equation}\label{vanishF}
 D_{k+1}^F + L_1^F  R_{1,k+1}^F + L_2^F R_{2,k+1}^F + \cdots + L_{k}^F R_{k,k+1}^F+  L_{k+1}^F,
 \end{equation}
respectively (note that both vanish). Let $2\leq i \leq k$. Then by (\ref{Ris}), it follows that
\begin{equation} 
 R_{i,k+1}^H=C(i,k+1,1) D^H(i,k+1,1) +\cdots +C(i,k+1,i_{k+1}) D^H(i,k+1,i_{k+1}), 
 \end{equation}
 and thus by (\ref{Riss}) 
 \begin{equation}
R_{i-1,k}^H= \tilde{C}(i,k+1,1)D^H(i,k+1,1) + \cdots +\tilde{C}{(i,k+1,i_{k+1})} D^H(i,k+1,i_{k+1}).
 \end{equation}
 Therefore, clearly for $1\leq j \leq i_{k+1}$, $D^H(i,k+1,j)=D^H(i-1,k,r_j)$, for some $1\leq r_j \leq {(i-1)}_{k}$. Similar conclusion holds for $D^F(i,k
 +1,j)$, i.e., for $1\leq j \leq i_{k+1}$, $D^F(i,k +1,j)=D^F(i-1,k,r_j)$.  Therefore, using ($\bullet \bullet$) for $s=k$, we have 
 $$
 D^H(i,k+1,j)=D^H(i-1,k,r_j)=\alpha^{-(k-i+1)}D^F(i-1,k,r_j)=\alpha^{-(k-i+1)} D^F(i,k+1,j),
 $$
 for $2\leq i \leq k$ and for $1\leq j \leq i_{k+1}$. 
 Therefore, 
\begin{equation}\label{RikF}
R_{i,k+1}^H= \alpha^{-(k+1-i)}R_{i,k+1}^F,
\end{equation}
for $2\leq i \leq k$. By  ($\bullet \bullet$), we have  $R_{1,k}^H = \alpha^{-(k-1)}R_{1,k}^F.$
 Thus comparing (\ref{R1sH}) and (\ref{R1s1H}) along with  using  the fact that $D_k^H=\alpha^{-(k+1)}D_k^F$, for $1\leq k \leq d-1$ (see (\ref{DHDF})), we obtain 
 \begin{align}\label{Ronek}
  R_{1,k+1}^H =\alpha^{-k} R_{1,k+1}^F.
  \end{align}
  Now as mentioned before  (\ref{vanishH}) and (\ref{vanishF}) both are equal to zero. Therefore using (\ref{RikF}), (\ref{Ronek}) and (\ref{DHDF}), we get 
  \begin{eqnarray*}
  L_{k+1}^H&=& -\left(D_{k+1}^H + L_1^H  R_{1,k+1}^H + L_2^H R_{2,k+1}^H + \cdots + L_{k}^H R_{k,k+1}^H\right)\\
  &=&-\alpha^{-(k+2)}\left(D_{k+1}^F + L_1^F  R_{1,k+1}^F + L_2^F R_{2,k+1}^F + \cdots + L_{k}^F R_{k,k+1}^F\right)\\
  &=&\alpha^{-(k+2)} L_{k+1}^F.
\end{eqnarray*} 
 Thus the claim follows.

\section{H\'{e}non maps with different Jacobians}
 \no 
 Let $H(x,y)=\left(y, p_H(y)-\delta_H x\right) $ and $F(x,y)=(y, p_F(y)-\delta_F x)$ be
  such that $\delta_H\neq \delta_F$.  As before, we assume that $\deg H=\deg F=d\geq 2$. It follows from 
 \cite[Thm.\ 2.7]{BRT} that if $I_H^+$ and $I_F^+$ are biholomorphic, then 
 \begin{equation} \label{degree}
 \delta_H^{d-1}=\delta_F^{d-1}=\delta.
 \end{equation}
 Now instead of working with $H$ and $F$, we work with $\tilde{H}=H^{d-1}$ and $\tilde{F}=F^{d-1}$. Note that since (\ref{degree}) holds the Jacobians of $\tilde{H}$ and $\tilde{F}$ are the same. Further, $I_{H^{d-1}}^+=I_{F^{d-1}}^+$.
 
 If $\tilde{H}: \mathbb{C}\times (\mathbb{C}\setminus \bar{\mathbb{D}}) \rightarrow \mathbb{C}\times (\mathbb{C}\setminus \bar{\mathbb{D}})$ is  the lift of $H:I_H^+ \rightarrow I_H^+$, then clearly ${\tilde{H}}^{d-1}: \mathbb{C}\times (\mathbb{C}\setminus \bar{\mathbb{D}}) \rightarrow \mathbb{C}\times (\mathbb{C}\setminus \bar{\mathbb{D}})$ is the lift of $H^{d-1}:I_H^+ \rightarrow I_H^+$. It follows from (\ref{H tilde}) that
 \[
 \tilde{H}(z,\zeta)=\left({(\delta_H}/{d}) z+Q_H(\zeta), \zeta^d \right).
 \]
 A simple calculation gives
 \[
{ \tilde{H}}^{d-1}(z,\zeta)=\left({\left({\delta_H}/{d}\right)}^{d-1}z+\tilde{Q}_H(\zeta), \zeta^{d^{d-1}}\right),
 \]
 where
$$
\tilde{Q}_H(\zeta)= {\left({\delta_H}/{d}\right)}^{d-2} Q_H(\zeta)+{\left({\delta_H}/{d}\right)}^{d-3} Q_H(\zeta^d)+\cdots +Q_H\left(\zeta^{d^{d-2}}\right).
$$
Note that  $\deg \tilde{H}=\deg \tilde{F}=\tilde{d}=d^{d-1}$ and $\det J_H=\det J_H=\delta$, where $J_H$ and $J_F$ are the Jacobian matrices of $H$ and $F$, respectively. Also it is easy to see that $\mathbb{Z}[{1}/{d^n}]=\mathbb{Z}[{1}/{d}]$ for all $n\geq 1$, and in particular, $\mathbb{Z}[{1}/{\tilde{d}}]=\mathbb{Z}[{1}/{d^{d-1}}]=\mathbb{Z}[{1}/{d}]$. As in (\ref{form aut}), one can show that 
for each element $\left[{k}/{\tilde{d}^n}\right] \in \mathbb{Z}[{1}/{\tilde{d}}]/\mathbb{Z}$, there exists a unique deck  transformation $\gamma_{{k}/{\tilde{d}^n}}$ of $\mathbb{C}\times (\mathbb{C}\setminus\bar{\mathbb{D}})$ such that 
\begin{equation}\label{form aut tilde}
\gamma_{{k}/{\tilde{d}^n}}
\begin{bmatrix}
z\\ \zeta
\end{bmatrix}
=\begin{bmatrix} 
 z+ \frac{\tilde{d}}{\delta} \sum_{l=0}^{n-1} {\left(\frac{\tilde{d}}{\delta}\right)}^l\left(\tilde{Q}_H(\zeta^{{\tilde{d}}^l})-\tilde{Q}_H\left( {\left(e^{\frac{2 k\pi i}{\tilde{d}^n}}\zeta)\right)}^{\tilde{d}^l} \right)\right)\\
 e^{\frac{2 k\pi i}{\tilde{d}^n}}\zeta
 \end{bmatrix},
\end{equation}
for each $n\geq 0$, $k\geq 1$ and for all $(z,\zeta)\in \mathbb{C}\times (\mathbb{C}\setminus\bar{\mathbb{D}})$. Thus, if $\tilde{p}=(z,\zeta)\in \mathbb{C}\times \left(\mathbb{C}\setminus \bar{\mathbb{D}}\right)$ is a point in the fiber of $p\in I_H^+=I_{\tilde{H}}^+$, then the other points in the fiber are precisely $\gamma_{{k}/{\tilde{d}^n}}(z,\zeta)$ with $n\geq 0$ and $k\geq 1$.

Note that the main hindrance to run a similar set of calculations as in Section 3 for any arbitrary pair of H\'{e}non maps $H$ and $F$ of the same degree is that the Jacobians of $H$ and $F$ are possibly different. But if we work with $H^{d-1}$ and $F^{d-1}$, there is no such issues.  Since (\ref{form aut tilde}) holds, a moment's thought assures that exactly similar  calculations as in Section 3 run smoothly for the maps $H^{d-1}$ and $F^{d-1}$.  Therefore, if $a: I_{\tilde{H}}^+ \rightarrow I_{\tilde{H}}^+$ is a biholomorphism and if $A: \mathbb{C}\times (\mathbb{C}\setminus\bar{\mathbb{D}}) \rightarrow \mathbb{C}\times (\mathbb{C}\setminus\bar{\mathbb{D}})$ is a lift, i.e, $A$ is of the form $(z,\zeta)\mapsto (\beta(z)\zeta +\gamma(\zeta), \alpha \zeta)$, then as in Section 3, we can show 

\begin{itemize}
\item 
$\beta(\zeta) \equiv \beta$ in $\mathbb{C}$;

\item 
$\left(\alpha^{d^{d-2}}\right)^{(d+1)d^{n-1}}\rightarrow \beta$ as $n\rightarrow \infty$ and consequently, $\alpha^{d^{d-2}(d+1)}=\beta$ with $\beta^{d-1}=1$.
\end{itemize} 
Thus as in (\ref{ReQHQF}) we obtain 
\begin{equation}\label{Re tild QHQF}
\beta \frac{\tilde{d}}{\delta} \left[\tilde{Q}_H(\zeta)-\tilde{Q}_H\left(e^{\frac{2\pi i}{\tilde{d}}} \zeta\right)\right]=\frac{\tilde{d}}{\delta} \left[\tilde{Q}_F(\alpha \zeta)-\tilde{Q}_F\left(e^{\frac{2\pi i}{\tilde{d}}} \alpha \zeta\right)\right],
 \end{equation}
for all $\zeta \in \mathbb{C}\setminus \bar{\mathbb{D}}$, with $\alpha^{d^{d-2}(d+1)}=\beta$ and $\beta^{d-1}=1$. Expanding (\ref{Re tild QHQF}), we get 
\begin{align}\label{QHQFtilde}
\beta \frac{\delta_H^{d-2}}{{{d}}^{d-2}}
\left[{Q}_H(\zeta)-{Q}_H  \left(e^{\frac{2 \pi i}{{\tilde{d}}}} \zeta \right)  \right]+\beta \frac{\delta_H^{d-3}}{{{d}}^{d-3}}\left[{Q}_H(\zeta^d)-{Q}_H\left( { \left(e^{\frac{2 \pi i}{{\tilde{d}}}} \zeta \right) }^d\right) \right]+ \nonumber \\
\cdots +\beta \left[{Q}_H\left(\zeta^{d^{d-2}}\right)-{Q}_H \left( {\left(e^{\frac{2 \pi i}{\tilde{d}}} \zeta \right)}^{d^{d-2}} \right) \right] \nonumber \\
= \frac{\delta_F^{d-2}}{{{d}}^{d-2}}
\left[{Q}_F(\alpha\zeta)-{Q}_F  \left(e^{\frac{2 \pi i}{{\tilde{d}}}}\alpha \zeta \right)  \right]+\frac{\delta_F^{d-3}}{{{d}}^{d-3}}\left[{Q}_F\left({(\alpha\zeta)}^d\right)-{Q}_F\left( { \left(e^{\frac{2 \pi i}{{\tilde{d}}}} \alpha \zeta \right) }^d \right)\right]+\nonumber \\
\cdots +\left[{Q}_F\left({(\alpha\zeta)}^{d^{d-2}}\right)-{Q}_F\left( {\left(e^{\frac{2 \pi i}{\tilde{d}}} \alpha \zeta \right)}^{d^{d-2}} \right) \right].
\end{align}
While comparing the coefficients of the polynomials appearing in the right hand side and in the left hand side of  (\ref{QHQFtilde}), note that if $\zeta^{d^r l}=\zeta^{d^s m}$ for $0\leq r,s \leq d-2$ and $1\leq l,m \leq d+1$ with $r\neq s$ (without loss of generality $r>s$, say), then $m=d^{r-s} l$. Since $d\geq 2$ and $0\leq r,s \leq d-2$ and $1\leq l,m \leq d+1$, clearly $r=s+1$ and $l=1$.  Thus since next to the highest coefficients of both $Q_H$ and $Q_F$ vanish ($Q_H$ and $Q_F$ are of the form (\ref{QH}) and (\ref{QF}), respectively), we obtain 
  \begin{equation}\label{QHQF comp}
  \beta\left[{Q}_H\left(\zeta^{d^{d-2}}\right)-{Q}_H\left( {\left(e^{\frac{2 \pi i}{\tilde{d}}} \zeta \right)}^{d^{d-2}} \right)\right]=
  {Q}_F\left({(\alpha\zeta)}^{d^{d-2}}\right)-{Q}_F\left( {\left(e^{\frac{2 \pi i}{\tilde{d}}} \alpha \zeta \right)}^{d^{d-2}} \right).
  \end{equation}
  Comparing both sides of (\ref{QHQF comp}), we get
   \begin{equation*}
 \beta A_{d-k}^H=\tilde{\alpha}^{d-k} A_{d-k}^F,
 \end{equation*}
 for $1\leq k \leq (d-1)$, where $\tilde{\alpha}=\alpha^{d^{d-2}}$. Equivalently, 
 we get 
 \begin{equation} \label{coeff Q mod}
 A_{d-k}^H=\tilde{\alpha}^{-(k+1)}A_{d-k}^F
 \end{equation}
  for $1\leq k \leq (d-1)$, with ${\tilde{\alpha}}^{d+1}=1$. Note that (\ref{coeff Q mod}) is an analogue of (\ref{coeffQ}) and therefore, using the same set of arguments as in Section 4, we get 
\begin{equation}\label{rel tild phpf}
\beta p_H(y)=\tilde{\alpha}p_F(\tilde{\alpha}y),
\end{equation}
for all $y\in \mathbb{C}$, with ${\tilde{\alpha}}^{d+1}=\beta$ and $\beta^{d-1}=1$.

\section{Proofs of Theorem \ref{main} and Theorem \ref{main'}}
\no 
{\it Proof of Theorem \ref{main}:}
It follows from (\ref{rel phpf}), (\ref{rel tild phpf}) and (\ref{degree}) that 
\begin{equation*}
\alpha p_F(\alpha y)=\beta p_H(y),
\end{equation*}
with $\alpha^{d+1}=1$ and 
\[
\delta_F=\gamma \delta_H,
\]
with $\alpha^{d^2-1}=1$ and $\beta^{d-1}=\gamma^{d-1}=1$. Thus if 
\begin{equation}\label{AB}
A(x,y)=(\alpha x, \beta \alpha^{-1} y)
\text{ and } B(x,y)=(\gamma \alpha \beta^{-1}x, \alpha^{-1}y),
\end{equation}
 then a simple calculation gives
\[
F\equiv A\circ H \circ B
\]
in $\mathbb{C}^2$.  Now note that if we write $A=(A_1,A_2)$ and $B=(B_1,B_2)$, then
\[
{\left(A_1(x,y)\right)}^{d-1}={(B_1(x,y))}^{d-1} \text{ and }{\left(A_2(x,y)\right)}^{d-1}={(B_2(x,y))}^{d-1}.
\]
Thus, 
\begin{equation}\label{ABdelta}
A_1(x,y)=\mu_1 B_1(x,y) \text{ and }A_2(x,y)=\mu_2 B_2(x,y),
\end{equation}
for some $\mu_1, \mu_2 \in \mathbb{C}$, with $\mu_1^{d-1}=\mu_2^{d-1}=1$. Now comparing (\ref{AB}) and (\ref{ABdelta}), we obtain $\mu_1=\gamma^{-1}\beta$ and $\mu_2=\beta$. Therefore, 
\begin{equation*} 
A(x,y)=L_\mu \circ B(x,y),
\end{equation*}
for all $(x,y)\in \mathbb{C}^2$, where $L_\mu(x,y)=(\mu_1 x, \mu_2 y)$. Thus, 
\begin{equation*}
F\equiv L_\mu \circ B \circ H \circ B.
\end{equation*}

\medskip 
\no 
{\it Proof of Theorem \ref{main'}:} Implementing the idea discussed just after stating the Theorem \ref{main} in the Introduction,  proof of Theorem \ref{main'} follows immediately once we prove Theorem \ref{main}.

\medskip 
\no 
It would be interesting to investigate the converse of Theorem \ref{main'}.

\medskip 
\no 
\it{Question:} If $H$ and $F$ are  two H\'{e}non maps such that 
\[
F=A_1\circ H \circ A_2
\] 
in $\mathbb{C}^2$, where $A_1$ and $A_2$ are affine automorphisms in $\mathbb{C}^2$, then how are the escaping sets of $H$ and $F$ related?


 \bibliographystyle{amsplain}

\end{document}